\documentclass{siamart171218}

\usepackage{amsmath}
\usepackage{amssymb}
\usepackage{graphicx}
\usepackage{arydshln}
\usepackage{hyperref}
\usepackage{enumitem}
\usepackage{algorithm, algpseudocode}
\usepackage{tikz}
\usetikzlibrary{decorations.pathreplacing}

\newcommand{\R}{\mathbb{R}}

\newcommand{\E}{\mathbb{E}}

\DeclareMathOperator*{\flopt}{fl}

\newtheorem{model}{Model}[section]

\numberwithin{theorem}{section}

\newcommand{\TheTitle}{A Refined Probabilistic Error Bound for Sums} 
\newcommand{\TheAuthors}{Eric Hallman}

\headers{Probabilistic Error Bound}{\TheAuthors}

\title{{\TheTitle}\thanks{This research was supported in part by the National Science Foundation through grant DMS-1745654.}}

\author{
  Eric Hallman \thanks{Department of Mathematics, North Carolina State University (\email{erhallma@ncsu.edu})}
}

\ifpdf
\hypersetup{
pdftitle={\TheTitle},
pdfauthor={\TheAuthors}
}
\fi

\begin{document}

\maketitle

\begin{abstract}
	This paper considers a probabilistic model for floating-point computation in which the roundoff errors are represented by bounded random variables with mean zero. Using this model, a probabilistic bound is derived for the forward error of the computed sum of $n$ real numbers. This work improves upon existing probabilistic bounds by holding to all orders, and as a result provides informative bounds for larger problem sizes. 
\end{abstract}



\section{Introduction}

It is generally known that deterministic error bounds for the computation of sums in finite precision are pessimistic in practice. Classical bounds grow proportionally to $nu$ where $n$ is the problem size and $u$ is the unit roundoff, but these bounds can be entirely uninformative when solving moderately large problems using low precision formats such as IEEE fp16 (where $u\approx 5\times 10^{-4}$) \cite{fp16} or bfloat16 (where $u\approx 4\times 10^{-3}$) \cite{bfloat16}. 

For decades a rule of thumb has suggested that the error will typically grow proportionally to $\sqrt{n}u$ instead. Higham and Mary \cite{higham2019new} show how to turn this rule of thumb into a rigorous error bound by making the useful (if not always realistic) assumption that the rounding errors in successive computations can be modeled as zero-mean independent random variables. They note that while the probabilistic bound in \cite{higham2019new} tends to be much closer to the actual error than the classical deterministic bounds are, it is still pessimistic by a factor of around $\sqrt{n}$ when the numbers to be summed are sampled from a zero-mean distribution. 

Higham and Mary derive a second bound in \cite{higham2020sharper} that both gives sharper results for random zero-mean data and weakens the assumption of independence of the rounding errors. Connolly, Higham, and Mary conclude in \cite{connolly2021stochastic} that this second probabilistic bound will hold unconditionally when using stochastic rounding, where by contrast it still provides only a rule of thumb for deterministic rounding. 

\subsection{Contributions}
In this article we refine the probabilistic bound of \cite{higham2020sharper} so that it holds to all orders. The earlier bound holds to first order only, and consequently works well when $nu<1$ but is less informative for larger $n$. Our refined version gives useful error bounds as long as $\lambda \sqrt{n}u < 1$, where $\lambda$ is a parameter that grows very slowly with $n$. 

The paper contains four main theorems.
\begin{itemize}
	\item Theorem \ref{thm:simpleBound} gives a structural result under the simplifying assumption that the rounding errors are independent random variables with mean zero. Theorem \ref{thm:simpleGeneral} uses this result to give {\it a priori} bounds under the assumption that the data to be summed are drawn randomly from a given interval. 
	\item Theorem \ref{thm:mainStructural} gives a slightly weaker structural bound under the more lenient assumption that the rounding errors are mean-independent rather than independent. Theorem \ref{thm:mainGeneral} uses this result to give {\it a priori} bounds in the same manner as Theorem \ref{thm:simpleGeneral}. 
	\end{itemize}
Theorems \ref{thm:mainStructural} and \ref{thm:mainGeneral} will hold unconditionally for stochastic rounding. 

Numerical experiments show that our bounds are pessimistic by about an order of magnitude, but accurately describe the growth rate of the forward error up to the point where $\lambda \sqrt{n}u\approx 1$. Beyond this point the sums computed using recursive summation have little relative accuracy in practice, and so our bounds suffice for practical purposes. 

\section{Background}
In this section we introduce notation, our models for the roundoff errors and data, and some useful tools from probability theory. 

\subsection{Notation}
Let $x_1,\ldots,x_n$ be the data to be summed. For each $k = 1:n$ we denote the exact and computed partial sums by $s_k = \sum_{i=1}^k x_i$ and  $\hat{s}_k = \flopt\left(\sum_{i=1}^k x_i\right)$, respectively. The computed sums $\hat{s}_k$ may be expressed through the recurrence
\[\hat{s}_k = \begin{cases}
	s_1 & k = 1, \\
	(\hat{s}_{k-1} + x_k)(1+\delta_k) & 2\leq k \leq n,
\end{cases}\]
where $\delta_k$ denotes the relative perturbation due to roundoff at the $k$-th step. For convenience, we let $\delta_1=0$. We define $E_k = \hat{s}_k - s_k$ to represent the forward error in the computation, and define ${\bf s}_k = [s_2,\ldots,s_k]$ for notational convenience.

Finally, we define the function $\lambda:(0,1)\rightarrow \R$ by 
\begin{equation}
	\lambda(\delta) = \sqrt{2\log(2/\delta)}.
\end{equation}
This expression appears frequently in our error bounds and illustrates how the probabilistic bound changes with respect to the desired failure probability $\delta$ (and possibly also the problem size). As noted in \cite{ipsen2020probabilistic},  $\lambda(\delta)$ grows very slowly as $\delta$ approaches zero. Even when $\delta = 10^{-16}$, for example, we would have $\lambda(\delta) \leq 9$. 

\subsection{Models for roundoff error and data}
We use the standard model for floating-point arithmetic \cite{higham2002accuracy}:
\begin{equation}\label{model:classical}
	\flopt(x \text{ op } y) = (x \text{ op } y)(1 + \delta), \quad |\delta| \leq u,\quad \text{ op } \in \{+,-,\times, /, \sqrt{}\},
\end{equation}
where $u$ is the unit roundoff. This model holds for IEEE arithmetic, with the exception of denormalized numbers. For the stochastic rounding system \cite{connolly2021stochastic}
\begin{equation}\label{eqn:stochRound}
	\flopt(x) = \begin{cases}
		\lceil x \rceil & \text{with probability } p = (x-\lfloor x\rfloor)/(\lceil x\rceil - \lfloor x\rfloor),\\
		\lfloor x \rfloor & \text{with probability } 1-p,
	\end{cases}
\end{equation}
the model will hold with $|\delta|\leq 2u$. Here $\lfloor x\rfloor$ and $\lceil x \rceil$ are respectively the largest and smallest floating-point numbers satisfying $\lfloor x\rfloor\leq x \leq \lceil x \rceil$.

Our first structural bound, given in Theorem \ref{thm:simpleBound}, relies on the following simplified model for rounding errors. 
\begin{model}\label{model:simple}
	Let the computation of interest generate rounding errors $\delta_1,\delta_2,\ldots$ in that order. The $\delta_k$ are independent random variables of mean zero satisfying $|\delta|\leq u$. 
\end{model}
When using stochastic rounding, Model \ref{model:simple} will hold with $|\delta_k|\leq 2u$ instead of $|\delta_k|\leq u$.

The second bound, given in Theorem \ref{thm:mainStructural}, uses a model that weakens the independence assumption to one of mean independence but tightens the boundedness condition. This new model still provides only a rule of thumb for deterministic roundoff errors, but it accurately describes the stochastic rounding system \eqref{eqn:stochRound}. 
\begin{model}\label{model:stochastic}
	Let the computation of interest generate rounding errors $\delta_1,\delta_2,\ldots$ in that order. The $\delta_k$ are random variables of mean zero satisfying \[\E(\delta_k | \delta_1,\ldots,\delta_{k-1}) = \E(\delta_k)\ (=0).\]
	Furthermore, the $\delta_k$ satisfy $a_k\leq \delta\leq b_k$, where $a_k$ and $b_k$ are functions of the data and of $\delta_1,\ldots,\delta_{k-1}$ satisfying $b_k - a_k\leq 2u$.
\end{model}
Since the errors from deterministic rounding satisfy $|\delta_k|\leq u$, the boundedness condition in Model \ref{model:stochastic} is satisfied with $a_k = -u$ and $b_k = u$. For stochastic rounding, by contrast, the bound $b_k-a_k\leq 2u$ is stronger than the constant bound $|\delta_k|\leq 2u$.

Theorems \ref{thm:mainGeneral} and \ref{thm:simpleGeneral} give {\it a priori} error bounds under the assumption that the data are independent random variables drawn from an interval. 
\begin{model}[probabilistic model of the data]\label{model:data}
	The $x_i$, $i=1:n$, are independent random variables sampled from a given distribution of mean $\mu_x$ and satisfy $|x_i-\mu_x|\leq C_x$, $i=1:n$, where $C_x$ is a constant. 
\end{model}
Theorem \ref{thm:mainGeneral} also uses a generalized version of Model \ref{model:stochastic} to include the assumption that the $\delta_k$ are mean independent of the data $x_i$ in addition to the previous perturbations. 
\begin{model}[probabilistic model of rounding errors for recursive summation]\label{model:refined}
	Consider the computation of $s_n=\sum_{i=1}^nx_i$ by recursive summation for random data $x_i$, satisfying Model \ref{model:data}. The rounding errors $\delta_2,\ldots,\delta_n$ produced by the computation are random variables of mean zero where for all $k$, the $\delta_k$ are mean independent of the previous rounding errors and the data, in the sense that
	\begin{equation}
		\E(\delta_k|\delta_2,\ldots,\delta_{k-1},x_1,\ldots,x_n) = \E(\delta_k)\ (=0). 
	\end{equation}
	Furthermore, the $\delta_k$ satisfy $a_k\leq \delta\leq b_k$, where $a_k$ and $b_k$ are both functions of $\delta_1,\ldots,\delta_{k-1}$ and $x_1,\ldots,x_n$ satisfying $b_k - a_k\leq 2u$.
\end{model}
Theorm \ref{thm:simpleGeneral} generalizes Model \ref{model:simple} similarly, assuming that the rounding errors are fully independent of each other and of the data.

Model \ref{model:simple} is also used in \cite{higham2019new}. Models \ref{model:stochastic} and \ref{model:refined} are similar to those used in \cite{higham2020sharper}, but with the condition $|\delta_k|\leq u$ replaced by the condition $a_k\leq \delta_k\leq b_k$ with $b_k-a_k\leq 2u$. Model \ref{model:data} is similar to the model for the data used in \cite{higham2020sharper}, but the condition $|x_i-\mu_x|\leq C_x$ replaces what was originally $|x_i|\leq C_x$. 

\subsection{Probability Theory}
We use the following definition of a martingale \cite{mitzenmacher2005probability}.

\begin{definition} [martingale] A sequence of random variables $Z_1,\ldots, Z_n$ is a martingale with respect to the sequence $X_1,\ldots,X_n$ if, for all $k\geq 1$,
	\begin{itemize}
		\item $Z_k$ is a function of $X_1,\ldots,X_k$, 
		\item $\E[|Z_k|]<\infty$, and 
		\item $\E(Z_{k+1}| X_1,\ldots, X_{k}) = Z_{k}$. 
	\end{itemize}
\end{definition}

We also use the following definition of a predictable process. 
\begin{definition}[predictable process]
	The sequence $A_1,\ldots, A_n$ is a predictable process with respect to the sequence $X_1,\ldots,X_n$ if, for all $k$, $A_k$ is a function of $X_1,\ldots,X_{k-1}$. 
\end{definition}
The sequences $\{a_k\}$ and $\{b_k\}$ in Models \ref{model:stochastic} and \ref{model:refined} are predictable processes with respect to the rounding errors and the data. 

The two definitions above are used in the following concentration bound \cite{roch2015modern}.
\begin{lemma}[Azuma-Hoeffding inequality] Let $Z_1,\ldots,Z_n$ be a martingale with respect to a sequence $X_1,\ldots, X_n$. Let $\{A_k\}_{k=2}^n$ and $\{B_k\}_{k=2}^n$ be predictable processes such that for all $k = 2:n$
	\[
		A_k \leq Z_{k} - Z_{k-1} \leq B_k\quad \text{and}\quad B_k - A_k \leq 2c_k
	\]z
	where the $c_k$ are constants. Then for any $\delta \in (0,1)$, with failure probability at most $\delta$,
	\begin{equation}\label{eqn:Azuma}
		|Z_n-Z_1| \leq  \lambda \|{\bf c}_n\|_2,
	\end{equation}
	where ${\bf c}_n = [c_2,\ldots,c_k]$ and $\lambda = \lambda(\delta) = \sqrt{2\log(2/\delta)}$. 
	\label{lemma:Azuma}
\end{lemma}
The Azuma-Hoeffding inequality admits two useful generalizations. 
\begin{itemize}
	\item It is noted in e.g.~\cite{sason2011refined, mcdiarmid1998concentration} that for the Azuma-Hoeffding inequality and several other concentration bounds the term $|Z_n - Z_1|$ may be replaced by $\max_{1\leq k \leq n}|Z_k - Z_1|$. We refer to this generalization as the {\it maximal} version of the Azuma-Hoeffding inequality, and it will sometimes allow us to avoid having to apply a union bound. 
	\item The bounds on the differences $Z_k-Z_{k-1}$ may fail with small probability and a similar but weaker concentration inequality will hold \cite{chung2006concentration}. Specifically, if the inequalities $B_k - A_k\leq 2c_k$ hold simultaneously for all $k = 2:n$ with total failure probability at most $\eta$, then \eqref{eqn:Azuma} will still hold with failure probability at most $\delta + \eta$.
\end{itemize}

\section{First structural bound}
We begin by deriving a probabilistic bound under Model \ref{model:simple}, which assumes independence of the roundoff errors. To do this, we use the following expression for the forward error. 
\begin{lemma}\label{lemma:forwardError}
	Let $s_n = \sum_{i=1}^n x_i$. The forward error from recursive summation satisfies
	\begin{equation}\label{eqn:forwardError}
		E_n = \hat{s}_n - s_n = \sum_{k=2}^ns_k\delta_k\prod_{j=k+1}^n(1+\delta_j). 
	\end{equation}
\end{lemma}
\begin{proof}
	From the model for roundoff error we have $\hat{s}_k = (\hat{s}_{k-1} + x_k)(1+\delta_k)$ for $k=2:n$. Subtracting $s_k$ from both sides gives
	\begin{align*}
		\hat{s}_k - s_k &= (\hat{s}_{k-1}-s_{k-1}+s_{k-1}+x_k)(1+\delta_k) - s_k\\
					&= (\hat{s}_{k-1} - s_{k-1} + s_k)(1+\delta_k) - s_k\\
					&= (\hat{s}_{k-1}-s_{k-1})(1+\delta_k) + s_k\delta_k. 
	\end{align*}
	The claim follows by unraveling the recurrence and using $\hat{s}_1 - s_1 = 0$. 
\end{proof}

A theorem by Higham and Mary \cite{higham2019new} allows us to bound the magnitude of the terms in \eqref{eqn:forwardError} with high probability. We present a modified version here. 
\begin{lemma}\label{lemma:deltaBound}
	Let $\delta_1,\ldots,\delta_n$ be independent random random variables of mean zero with $|\delta_k|\leq u$. Then for any $\delta\in(0,1)$, with failure probability at most $\delta$, 
	\[
		\max_{1\leq k \leq n} \left|\prod_{j=k+1}^n(1+\delta_j) - 1\right| \leq \tilde{\gamma}_n(\delta),
	\]
	where 
	\[
		\tilde{\gamma}_n(\delta) := \exp\left(\frac{\lambda(\delta)\sqrt{n}u + nu^2}{1-u}\right) - 1. 
	\]
\end{lemma}
\begin{proof}
	Modify the proof of Theorem 4.6 in \cite{connolly2021stochastic} to let $Z_k = \sum_{j=n-k+2}^n\delta_j$ for $k=1:n$, and replace the use of the Azuma-Hoeffding inequality with its maximal version. 
\end{proof}

We now have the tools we need to put a probabilistic bound on the forward error. 
\begin{theorem}\label{thm:simpleBound}
	Let $s_n = \sum_{i=1}^nx_i$ and let $\hat{s}_n$ be computed by recursive summation, and assume that Model \ref{model:simple} holds. Then for any $\delta\in(0,1)$, with failure probability $\delta$, 
	\begin{equation}
		|E_n| \leq u\|{\bf s}_n\|\lambda(\delta/2)(1+\tilde{\gamma}_n(\delta/2)). 
	\end{equation}
\end{theorem}
\begin{proof}
	For $k = 1:n-1$ define the random variables 
	\[Z_k = \sum_{j=n-k+1}^ns_j\delta_j\prod_{\ell=j+1}^n(1+\delta_\ell).\]
	Then $Z_{n-1} = E_n$ by Lemma \ref{lemma:forwardError}, and $Z_1,\ldots,Z_{n-1}$ is a martingale with respect to the random variables $\delta_n,\ldots,\delta_2$. Furthermore, by Lemma \ref{lemma:deltaBound} the bounds
	\begin{equation}\label{eqn:gammTildeBounds}
		|Z_k - Z_{k-1}| = \left|s_{n-k+1}\delta_{n-k+1}\prod_{\ell=n-k+2}^n(1+\delta_\ell)\right| \leq u|s_{n-k+1}|(1+\tilde{\gamma}_n(\delta/2))
	\end{equation}
	all hold simultaneously with failure probability at most $\delta/2$. The claim follows applying the Azuma-Hoeffding inequality with failure probability $\delta/2$. 
\end{proof}

\subsection{Commentary} \label{sec:simpleCommentary}
We make a few observations about Theorem \ref{thm:simpleBound}. First, the theorem will hold for stochastic rounding with the substitution $u \leftarrow 2u$. Second, the term $1+\tilde{\gamma}_n(\delta/2)$ remains close to 1 as long as $\lambda(\delta/2)\sqrt{n}u < 1$, but for larger problem sizes it increases exponentially. One interpretation of this observation is that if a sufficient majority of rounding operations are away from zero (i.e., $\delta_k > 0$), the error in the computed sum will grow rapidly. 


Finally, we emphasize that the proof of the theorem relies on the assumption that the perturbations are independent. Under the broader Model \ref{model:stochastic} which assumes only mean-independence of the perturbations, the sequence $Z_1,\ldots,Z_{n-1}$ would not necessarily be a martingale. As of the time of writing, the author has been unable to find a modification to the proof that would accommodate Model \ref{model:stochastic}. A different approach is therefore in order. 

\section{Second structural bound}
In order to obtain a probabilistic bound that holds under Model \ref{model:stochastic}, we introduce a new expression for the forward error. Consider the expression for $E_n$ in \eqref{eqn:forwardError} as a multivariate polynomial with respect to the variables $\delta_2,\ldots,\delta_n$, and rewrite it as
\begin{equation}\label{eqn:forwardErrorBinomial}
	E_n = \sum_{j=1}^{n-1}S_n^{(j)}, 
\end{equation}
where each $S_n^{(j)}$ is the sum of all terms of order $j$. From the proof of Lemma \ref{lemma:forwardError} we observe that the forward error satisfies the recurrence $E_k = E_{k-1}(1+\delta_k) + \delta_ks_k$ for $k=2:n$. By equating terms of order $j$ on both sides, we get
\[
	S_k^{(j)} = \begin{cases}
		0 & k = 1, \\
		S_{k-1}^{(j)} + \delta_ks_k & k > 1,\ j = 1 \\
		S_{k-1}^{(j)} + \delta_kS_{k-1}^{(j-1)} & k > 1,\ j > 1. 
	\end{cases}
\]
Note that $S_k^{(j)}=0$ whenever $j\geq k$ since the error in the sum of $k$ numbers contains no terms of order greater than $k-1$. 

By unraveling the above recurrence, we obtain an identity analogous to the hockey-stick identity for binomial coefficients. 
\[
	S_k^{(j)} = \begin{cases}
		\sum_{i = 2}^k \delta_is_i & j = 1, \\
		\sum_{i=2}^k \delta_iS_{i-1}^{(j-1)} & j > 1.\\
	\end{cases}
\]
The key idea in the proof of Theorem \ref{thm:mainStructural} is that for each $j=1:n-1$, the sequence of random variables $S_1^{(j)},\ldots,S_n^{(j)}$ is a martingale with respect to $\delta_1,\ldots,\delta_n$ under Model \ref{model:stochastic} because each term $S_{i-1}^{(j-1)}$ is a function of $\delta_1,\ldots,\delta_{i-1}$. Given a bound on the terms of order $j-1$, we can therefore use the Azuma-Hoeffding inequality to bound the terms of order $j$. By repeating this process for each $j=1:n-1$ and using the identity \eqref{eqn:forwardErrorBinomial}, we obtain the following bound on the forward error.

\begin{theorem}\label{thm:mainStructural}
	Let $s_n = \sum_{i=1}^nx_i$ and let $\hat{s}_n$ be computed by recursive summation, and assume that Model \ref{model:stochastic} holds. Then for any $\delta\in(0,1)$, with failure probability $\delta$, 
	\[|E_n| \leq \left(\frac{1-\kappa^{n-1}}{1-\kappa}\right)
	\lambda \|{\bf s}_n\|_2u. 
	\]
	where $\lambda :=  \lambda(\delta/(n-1)) = \sqrt{2\log(2(n-1)/\delta)}$ and $\kappa := \kappa(n,\delta) = \lambda\sqrt{n}u$.
\end{theorem}
\begin{proof}
	Begin with the first-order error term $S_n^{(1)} = \sum_{i=2}^n\delta_is_i$. By applying the maximal Azuma-Hoeffding inequality, it follows that with failure probability at most $\delta/(n-1)$, 
	\begin{equation}\label{eqn:firstOrder}
		\max_{1\leq k\leq n} |S_k^{(1)}| \leq u\|{\bf s}_n\|_2\sqrt{2\log(2(n-1)/\delta)} =: \beta. 
	\end{equation}
	Continue to the second-order term $S_n^{(2)} = \sum_{i=2}^n\delta_iS_{i-1}^{(1)}$. The bound \eqref{eqn:firstOrder} holds with failure probability at most $\delta/(n-1)$, so by applying the maximal Azuma-Hoeffding inequality with failure probability $\delta/(n-1)$ it follows that with failure probability at most $2\delta/(n-1)$, 
	\begin{align}\label{eqn:secondOrder}
		\begin{split}
		\max_{1\leq k \leq n} |S_k^{(2)}| &\leq u \left(\sum_{i=2}^n\beta^2\right)^{1/2}\sqrt{2\log(2(n-1)/\delta)}\\
		&\leq u\beta\sqrt{n}\sqrt{2\log(2(n-1)/\delta)} \\
		&= \kappa \beta. 
		\end{split}
	\end{align}
	More than that, it follows that with failure probability at most $2\delta/(n-1)$ the bounds \eqref{eqn:firstOrder} and \eqref{eqn:secondOrder} hold simultaneously. 
	
	Continuing in this manner for $j=3,\ldots,n-1$, we conclude that with failure probability at most $\delta$ the bound
	\[
		\max_{1\leq k\leq n} |S_k^{(j)}| \leq \kappa^{j-1}\beta
	\]
	holds for all $1\leq j \leq n-1$ simultaneously. The claim follows by using the triangle inequality on the expression for the forward error in \eqref{eqn:forwardErrorBinomial}. 
\end{proof}

\section{Probabilistic bound for random data}
Following the lead of Higham and Mary \cite{higham2020sharper}, we derive an error bound that does not directly depend on the partial sums $s_k$ by making some assumptions about the distribution of the data $x_i$. These assumptions, and the modified model for rounding errors, are given in Models \ref{model:data} and \ref{model:refined}. In short, the data are assumed to be independent random variables satisfying $|x_i-\mu_x|\leq C_x$ for constants $\mu_x$ and $C_x$, and the perturbations $\delta_k$ are assumed to be mean-independent of the previous perturbations $\delta_1,\ldots,\delta_{k-1}$ as well as the data $x_1,\ldots,x_n$.

By applying the maximal Azuma-Hoeffding inequality to the variables $Z_1=0$, $Z_{k+1} = s_k - k\mu_x$ for $k = 1:n$, we find that with failure probability at most $\delta/n$, 
\[
	\max_{1\leq k\leq n} |s_k - k\mu_x| \leq \sqrt{n}C_x\sqrt{2\log(2n/\delta)},
\]
and therefore that with failure probability at most $\delta/n$, 
\[
	\max_{1\leq k\leq n} |s_k| \leq k|\mu_x| + \sqrt{n}C_x\sqrt{2\log(2n/\delta)},
\]
which implies that with failure probability at most $\delta/n$, 
\begin{equation}\label{eqn:sbound}
	\|{\bf s}_n\|_2 \leq n^{3/2}|\mu_x| + nC_x\sqrt{2\log(2n/\delta)}. 
\end{equation}
By inserting this probabilistic bound into Theorem \ref{thm:mainStructural} applied with failure probability $(n-1)\delta/n$, we arrive at the following {\it a priori} probabilistic bound. 
\begin{theorem}\label{thm:mainGeneral}
	Let $s_n = \sum_{i=1}^nx_i$ and let $\hat{s}_n$ be computed by recursive summation, and assume that the data and perturbations satisfy Model \ref{model:refined}. Then for any $\delta\in(0,1)$, with failure probability at most $\delta$, 
	\[|E_n| \leq \left(\frac{1-\kappa^{n-1}}{1-\kappa}\right)\left(\lambda|\mu_x|n^{3/2} + \lambda^2C_xn\right)
	u, 
	\]
	where $\lambda := \lambda(\delta/n) = \sqrt{2\log(2n/\delta)}$ and 		$\kappa := \kappa(n,\delta) = \lambda\sqrt{n}u$. 
\end{theorem}

This theorem closely resembles Theorem 2.8 of \cite{higham2020sharper}, but differs notably in that it holds to all orders where the earlier theorem holds only to first order. A more minor difference is that the term $C_x$ bounds $|x_i - \mu_x|$ where in the earlier theorem it bounded $|x_i|$, so for the same set of data Theorem \ref{thm:mainGeneral} may admit a smaller constant $C_x$. 

In a similar manner, we can combine \eqref{eqn:sbound} with Theorem \ref{thm:simpleBound} to derive a tighter error bound under the assumption that the $\delta_k$ are independent of each other and the data. By taking \eqref{eqn:sbound} with failure probability $\delta/3$ and Theorem \ref{thm:simpleBound} with failure probability $2\delta/3$, we get the following result. Note that as with Theorem \ref{thm:simpleBound}, it requires the substitution $u\leftarrow 2u$ to apply to stochastic rounding. 
\begin{theorem}\label{thm:simpleGeneral}
		Let $s_n = \sum_{i=1}^nx_i$ and let $\hat{s}_n$ be computed by recursive summation. Assume that the data satisfies Model \ref{model:data} and that the perturbations satisfy Model \ref{model:simple} and are furthermore independent of the data. Then for any $\delta\in(0,1)$, with failure probability at most $\delta$, 
	\[|E_n| \leq (1+\tilde{\gamma}_n(\delta/3))\left(\lambda|\mu_x|n^{3/2} + \lambda^2C_xn\right)
	u, 
	\]
	where $\lambda := \lambda(\delta/3) = \sqrt{2\log(6/\delta)}$. 
\end{theorem}


\section{Numerical experiments} 
Here we present several numerical experiments to test our probabilistic bounds. The experiments are done in Matlab R2021a, and half precision (fp16) and bfloat16 are simulated using the Matlab function \texttt{chop}\footnote{\url{https://github.com/higham/chop}} as implemented by Higham and Pranesh in \cite{higham2019simulating}. The \texttt{chop} function can simulate both deterministic and stochastic rounding. 

\subsection{Error bounds}
In the first set of experiments, we examine how well the bounds of Theorems \ref{thm:mainGeneral} and \ref{thm:simpleGeneral} describe the error in practice. For random uniform $[-1,1]$ data, we compute $s_n = \sum_{i=1}^nx_i$ using recursive summation in double precision, half precision, and the bfloat16 format. The forward errors are then computed using the double precision computation as the exact answer. 

The results of 50 trials are shown for fp16 in Figure \ref{fig:half} and for bfloat16 in Figure \ref{fig:bfloat}. The probabilistic error bounds of Theorems \ref{thm:mainGeneral} (assuming mean-independence of the roundoffs) and \ref{thm:simpleGeneral} (assuming full independence) are shown for failure probability $\delta=0.05$. As was the case for the error bounds of \cite{higham2020sharper}, our bounds are pessimistic and appear to hold in practice with $\lambda \approx 1$. The errors under stochastic rounding are not notably different from those under deterministic rounding even though the error bound of Theorem \ref{thm:simpleGeneral} is more pessimistic for stochastic rounding due to the substitution $u \leftarrow 2u$. If anything, the worst-case errors for stochastic rounding appear run slightly smaller than their deterministic counterparts. 

Because the new bounds hold to all orders, however, they establish that the forward error can be expected to grown in a controlled manner at least until $\lambda\sqrt{n}u \approx 1$. Taking $\lambda \approx 9$, this bound corresponds to the problem sizes $n\approx 3.5\times 10^{12}$ for single precision, $n\approx 5.2\times 10^{4}$ for fp16, and $n\approx 810$ for bfloat16. The precise behavior of the error beyond this point is something of a moot question: Higham and Mary note in \cite{higham2020sharper} that the relative forward error of the computed sum will typically grow proportionally to $\sqrt{n}u$. Thus if $\sqrt{n}u > 1$ the computed sum may not have any relative accuracy at all.

\begin{figure}		\label{fig:half}
    \centering
    \begin{minipage}{0.49\textwidth}
        \centering
        \includegraphics[trim={1cm .2cm 1.3cm .2cm},clip,width=\textwidth]{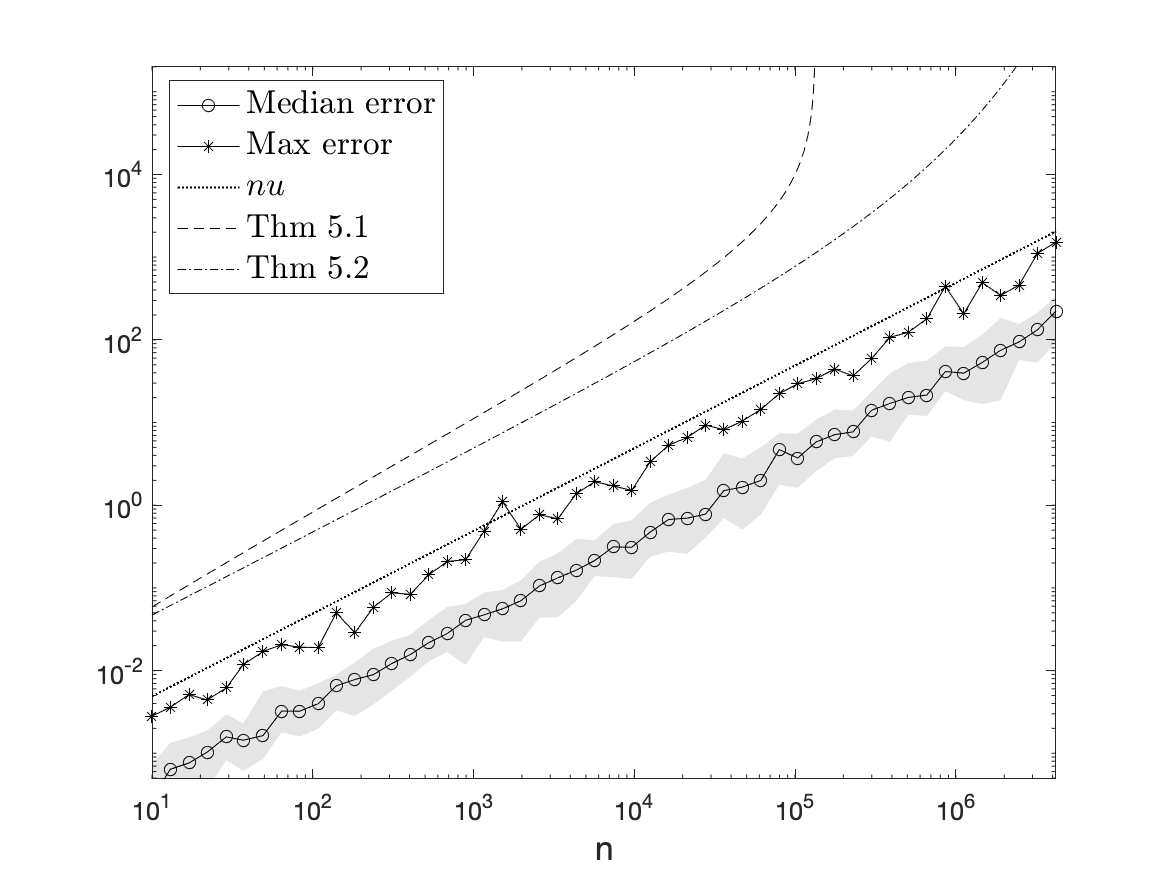}
    \end{minipage}\hspace{0.1cm}
    \begin{minipage}{0.49\textwidth}
        \centering
        \includegraphics[trim={1cm .2cm 1.3cm .2cm},clip,width=\textwidth]{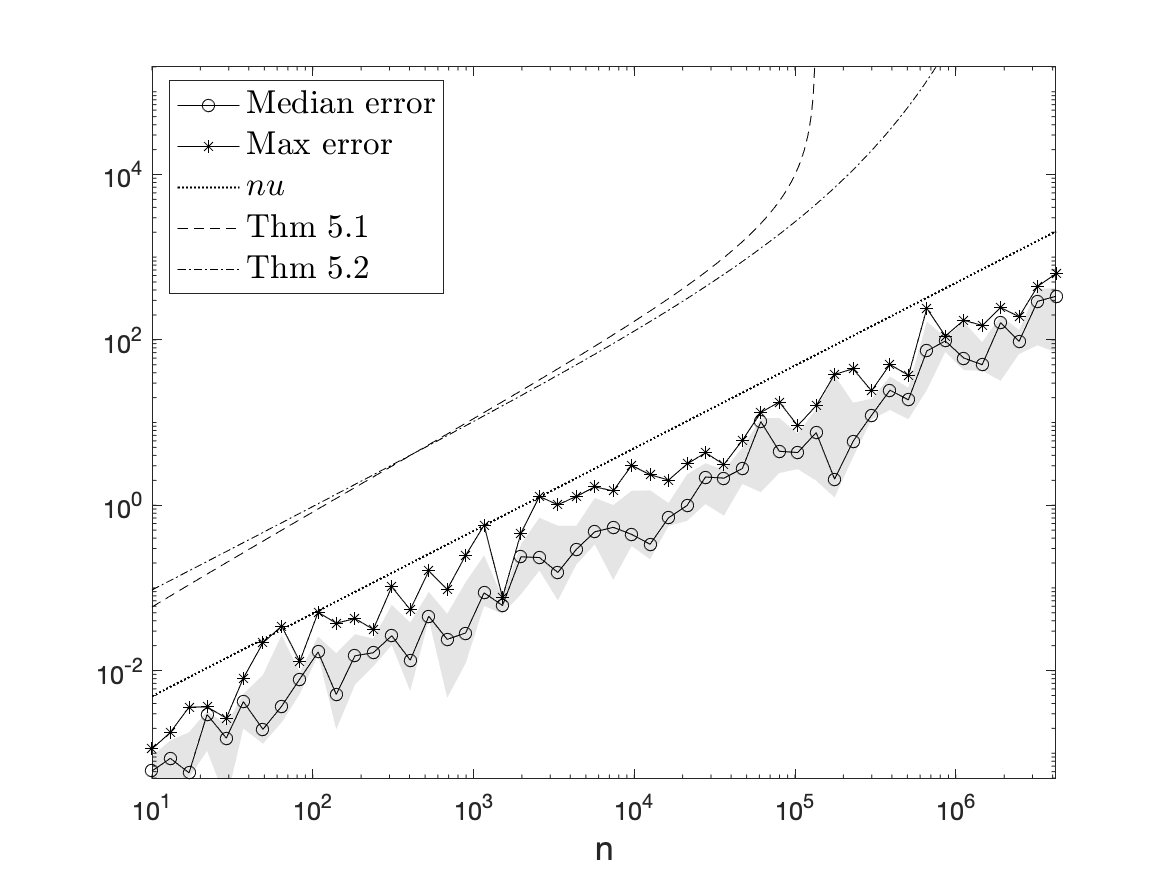}
    \end{minipage}
	\caption{Absolute forward error bounds for the half precision computation of $\sum_{i=1}^nx_i$ with random uniform $[-1,1]$ data. The shaded regions enclose the 25th and 75th percentile errors over 50 trials with $\delta = 0.05$. Left: Deterministic rounding. Right: Stochastic rounding. }
\end{figure}

\begin{figure}		\label{fig:bfloat}
    \centering
    \begin{minipage}{0.49\textwidth}
        \centering
        \includegraphics[trim={1cm .2cm 1.3cm .2cm},clip,width=\textwidth]{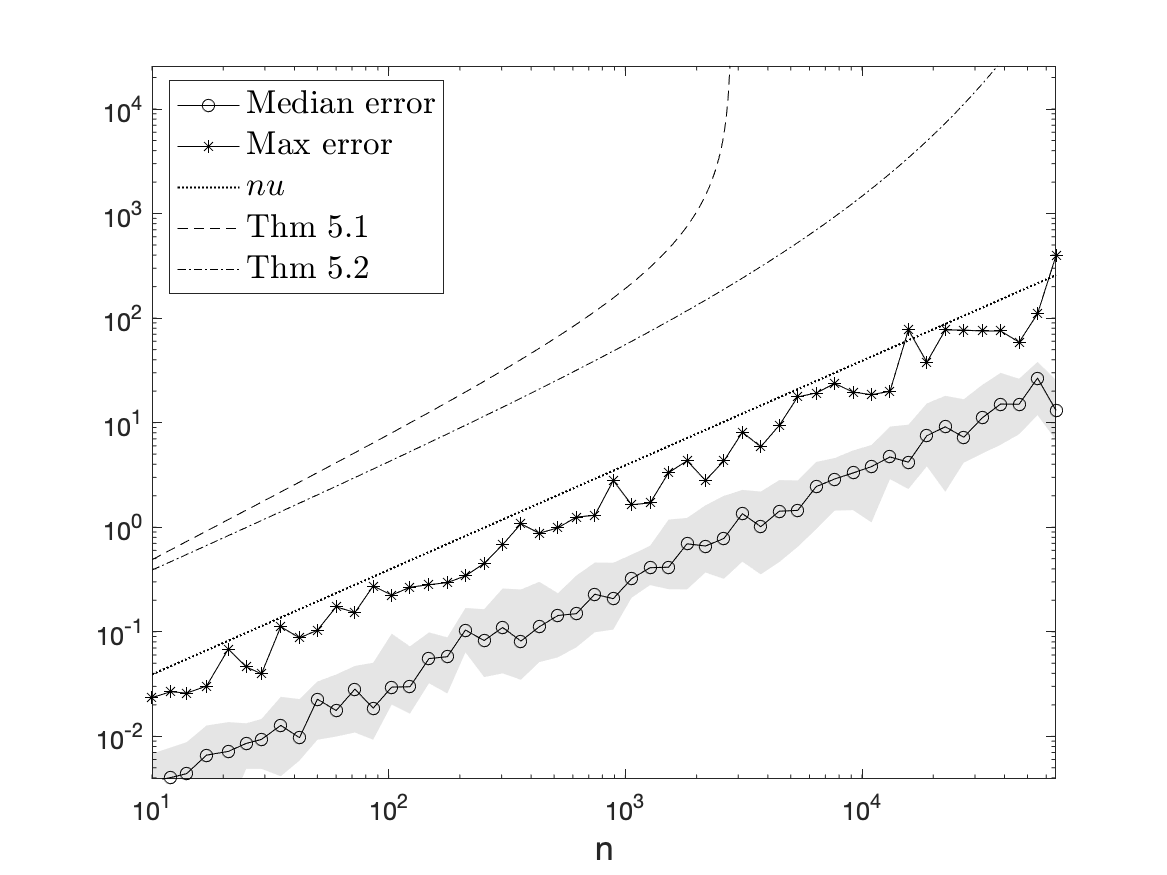}
    \end{minipage}\hspace{0.1cm}
    \begin{minipage}{0.49\textwidth}
        \centering
        \includegraphics[trim={1cm .2cm 1.3cm .2cm},clip,width=\textwidth]{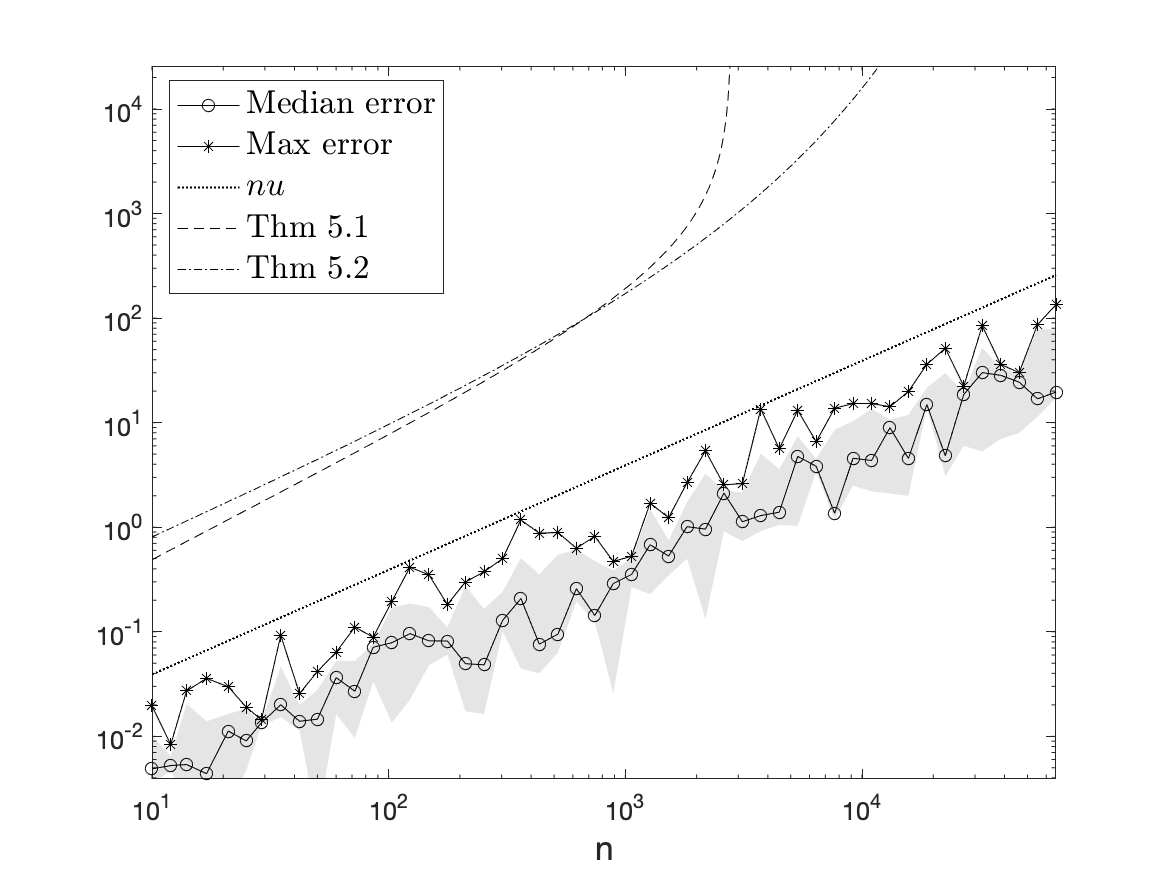}
    \end{minipage}
	\caption{Absolute forward error bounds for the bfloat16 computation of $\sum_{i=1}^nx_i$ with random uniform $[-1,1]$ data. The shaded regions enclose the 25th and 75th percentile errors over 50 trials with $\delta = 0.05$. Left: Deterministic rounding. Right: Stochastic rounding. }
\end{figure}

\subsection{Cumulative rounding errors} 
The previous experiments test problem sizes as large as as $n= u^{-2}$, or equivalently to the point where $\sqrt{n}u=1$. It is interesting that at least up to this point the forward errors do not begin to grow rapidly in the manner that the error bounds of Theorems \ref{thm:mainGeneral} and \ref{thm:simpleGeneral} would predict. In Section \ref{sec:simpleCommentary}, we note that the errors might be expected to grow rapidly if a sufficient majority of the rounding operations are away from zero, in which case the product $\prod_{i=1}^n(1+\delta_i)$ would grow exponentially. 

In the second set of experiments, we examine how the product $\prod_{i=1}^n(1+\delta_i)$ grows in practice and compare it to the bounds of Lemma \ref{lemma:deltaBound}. For random uniform $[-1,1]$ data, we compute $s_n = \sum_{i=1}^nx_i$ using recursive summation in blfoat16. At each index $i$, we record the size of the perturbation $\delta_i$ by comparing $\hat{s}_{i} = \flopt(\hat{s}_{i-1} + x_i)$ to the value where $\hat{s}_{i-1}$ is computed using bfloat16 but the addition $\hat{s}_{i-1}+x_i$ is done in single precision. The problem size $n$ ranges from $10$ to $3\times 10^5\approx 4.5u^{-2}$, and 50 trials are run for each value of $n$. 

Results are shown in Figure \ref{fig:deltaErr}, with a vertical dashed line representing the point where $\sqrt{n}u =1$. To the right of this line, the relative error will typically be close to 1 and so the precise behavior of the error is of less practical importance. As with the previous experiments, the bounds of Lemma \ref{lemma:deltaBound} are pessimistic and appear to hold in practice with $\lambda \approx 1$. Unlike the previous experiments, the errors behave quite differently depending on whether deterministic or stochastic rounding is used. The discrepancy is likely due to stagnation: once the computed partial sums $\hat{s}_i$ exceed $u^{-1}$ in magnitude, the deterministic rounding errors become biased toward zero. When using stochastic rounding the errors remain unbiased and so exponential growth in the product $\prod_{i=1}^n(1+\delta_i)$ is still possible, but the median product starts to decline because rounding toward zero becomes more likely than rounding away from zero. 

In summary, around the point when $\sqrt{n}u\approx 1$ we observe the following phenomena: 
\begin{itemize}
	\item The probabilistic bounds of this article grow rapidly and fail to capture the actual behavior of the errors. 
	\item Due to stagnation, the forward error for deterministic rounding will not grow exponentially. Exponential growth in the error may still occur for stochastic rounding, but if so it happens significantly later than our bounds would predict. 
	\item The relative forward error in the computed sum approaches 1 for a typical trial, for data with either zero or non-zero mean. It is therefore not recommended to compute sums of size $n\approx u^{-2}$ or larger using recursive summation. 
\end{itemize}

\begin{figure}		\label{fig:deltaErr}
    \centering
    \begin{minipage}{0.49\textwidth}
        \centering
        \includegraphics[trim={1cm .2cm 1.3cm .2cm},clip,width=\textwidth]{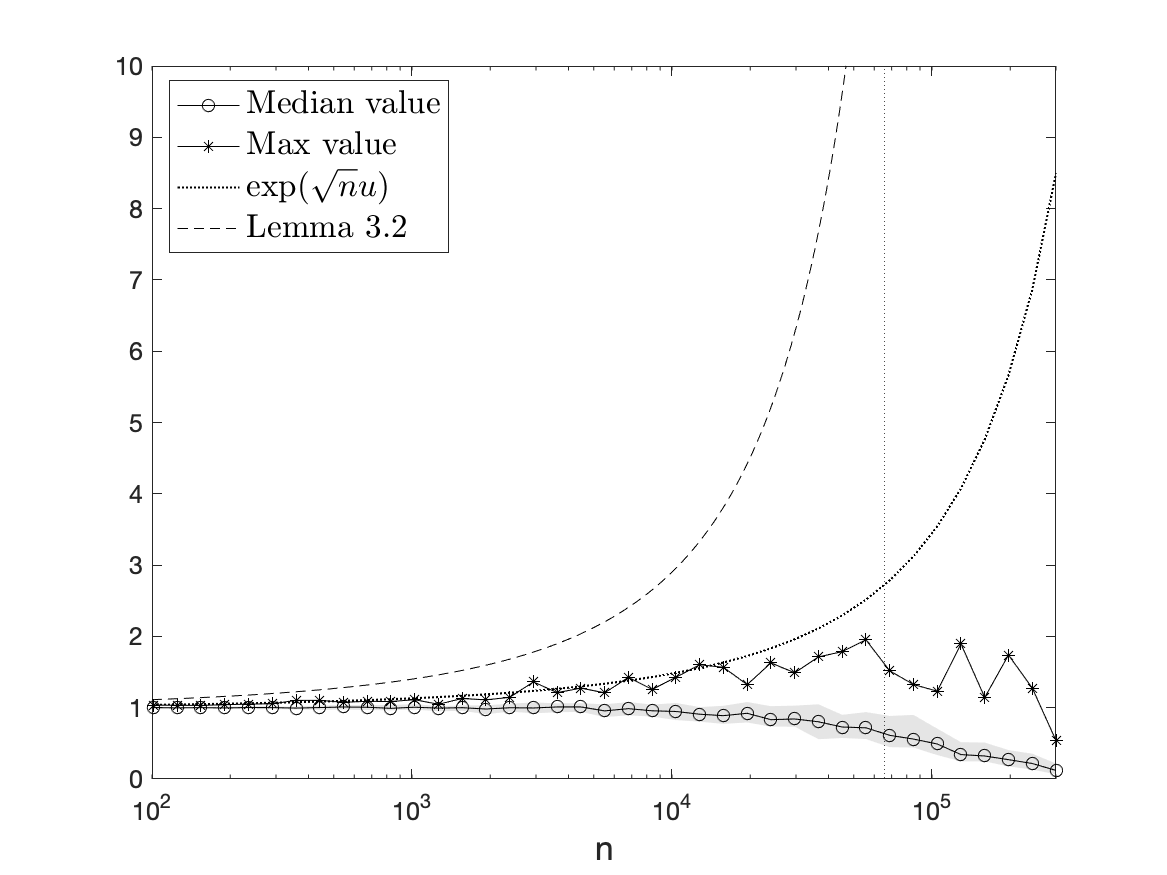}
    \end{minipage}\hspace{0.1cm}
    \begin{minipage}{0.49\textwidth}
        \centering
        \includegraphics[trim={1cm .2cm 1.3cm .2cm},clip,width=\textwidth]{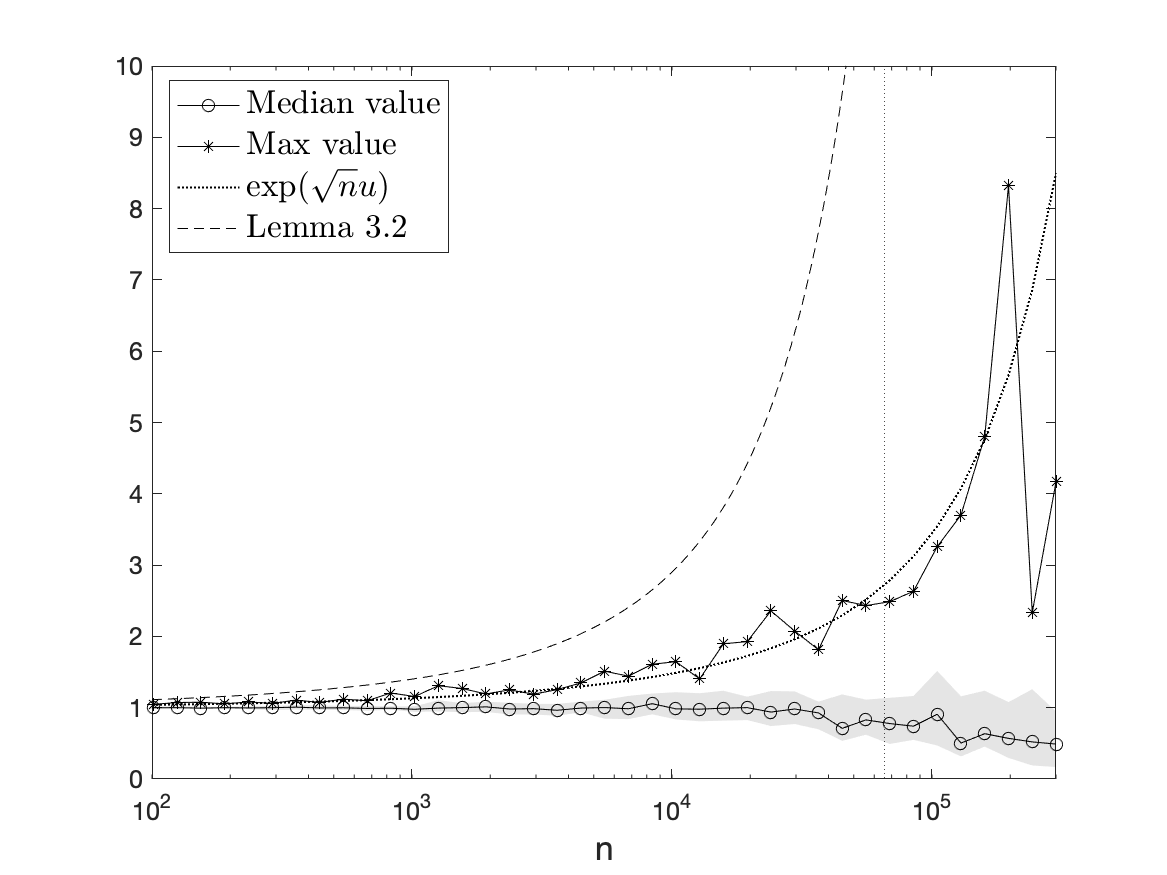}
    \end{minipage}
	\caption{Products $\prod_{i=1}^n(1+\delta_i)$ from the bfloat16 computation of $\sum_{i=1}^nx_i$ with random uniform $[-1,1]$ data. The shaded regions enclose the 25th and 75th percentile errors over 50 trials with $\delta = 0.05$. Left: Deterministic rounding. Right: Stochastic rounding. }
\end{figure}

\section{Conclusions}
We have introduced two probabilistic bounds on the forward error of a sum computed using recursive summation. Our bounds are made explicit to all orders and hold up to problem sizes $n$ where $\lambda \sqrt{n}u\approx 1$, where $u$ is the unit roundoff and $\lambda$ a small constant factor. Although the bounds appear to be pessimistic in practice by about an order of magnitude, they accurately describe the growth rate of the error for problem sizes of practical interest. The model of roundoff error used for one of these bounds accurately describes the behavior of at least one type of stochastic rounding, and so the associated bound may be taken not just as a rule of thumb for the behavior of the error under deterministic rounding, but as a rigorous bound on the growth of the forward error when using stochastic rounding.

\bibliographystyle{siamplain}
\bibliography{references}

\end{document}